\documentclass[12pt]{article}
\usepackage{amssymb,amsmath}
\usepackage{amsthm}
\usepackage{cases}
\usepackage{amsfonts}
\usepackage{graphicx}
\usepackage{cite,color,xcolor}
\usepackage[margin=1 in]{geometry}
\usepackage[colorlinks,citecolor=blue,urlcolor=blue]{hyperref}
\usepackage[utf8]{inputenc}
\usepackage[pagewise]{lineno}

\newtheorem{theorem}{Theorem}[section]

\newtheorem{lemma}{Lemma}[section]
\newtheorem{proposition}{Proposition}[section]

\newtheorem{definition}{Definition}[section]
\newtheorem{remark}{Remark}[section]

\usepackage{txfonts}

\newcommand{\bal}{\begin{align}}
\newcommand{\bbal}{\begin{align*}}
\newcommand{\beq}{\begin{equation}}
\newcommand{\eeq}{\end{equation}}
\newcommand{\bca}{\begin{cases}}
\newcommand{\eca}{\end{cases}}

\newcommand{\pa}{\partial}
\newcommand{\fr}{\frac}
\newcommand{\na}{\nabla}
\newcommand{\De}{\Delta}

\newcommand{\cd}{\cdot}
\newcommand{\ep}{\varepsilon}
\newcommand{\dd}{\mathrm{d}}

\newcommand{\R}{\mathbb{R}}
\newcommand{\T}{\mathbb{T}}

\newcommand{\f}{\left}
\newcommand{\g}{\right}

\begin{document}

\title{Non-convergence of the Navier-Stokes equations toward the Euler equations in the endpoint Besov spaces}

\author{Yanghai Yu$^{1,}\footnote{E-mail: yuyanghai214@sina.com(Corresponding author); lijinlu@gnnu.edu.cn}$ and Jinlu Li$^{2}$\\
\small $^1$ School of Mathematics and Statistics, Anhui Normal University, Wuhu 241002, China\\
\small $^2$ School of Mathematics and Computer Sciences, Gannan Normal University, Ganzhou 341000, China}

\date{\today}
\maketitle
\begin{abstract}
In this paper, we consider the inviscid limit problem to the higher dimensional incompressible Navier-Stokes equations in the whole space. It was proved in \cite[J. Funct. Anal., 276 (2019)]{GZ} that given initial data $u_0\in B^{s}_{p,r}$ with $1\leq r<\infty$, the solution of the Navier-Stokes equations converges strongly in $B^{s}_{p,r}$ to the solution of the Euler equations as the viscosity parameter tends to zero. In the case when $r=\infty$, we prove the failure of the $B^{s}_{p,\infty}$-convergence of the Navier-Stokes equations toward the Euler equations in the inviscid limit.
\end{abstract}

{\bf Keywords:} Navier-Stokes equations, Euler equations, Inviscid limit, Besov spaces.

{\bf MSC (2010):} 35Q30, 76B03.
\vskip0mm

\section{Introduction}
The motion of an incompressible viscous fluid with constant density is governed
by the following Navier-Stokes equations
\begin{align}\label{1}\tag{NS}
\begin{cases}
\pa_t u+u\cdot \nabla u-\ep \Delta u+\nabla P=0,\\
\mathrm{div\,} u=0,\\
u(0,x)=u_0(x),
\end{cases}
\end{align}
where $\ep>0$ denotes the viscosity of the fluid, the vector field $u(t,x):[0,\infty)\times {\mathbb R}^d\to {\mathbb R}^d$ stands for the velocity of the fluid, the quantity $P(t,x):[0,\infty)\times {\mathbb R}^d\to {\mathbb R}$ denotes the scalar pressure, and $\mathrm{div\,} u=0$ means that the fluid is incompressible.

When the viscocity vanishes ($\ep=0$), the Navier-Stokes equations \eqref{1} reduces to the Euler equations for ideal incompressible fluid
\begin{align}\label{2}\tag{E}
\begin{cases}
\pa_t u+u\cdot \nabla u+\nabla P=0, \\
\mathrm{div\,} u=0,\\
u(0,x)=u_0(x).
\end{cases}
\end{align}
The mathematical study of both the Navier-Stokes equations \eqref{1} and Euler equations \eqref{2} has a long and distinguished history (see Constantin's survey \cite{P} for more details). We do not detail the literature since it is huge and refer the readers to see the monographs of Majda-Bertozzi \cite{Majda01} and Bahouri-Chemin-Danchin \cite{B.C.D} for the well-posedness results of both Navier-Stokes and Euler equations in Sobolev and Besov spaces respectively.

A classical problem in fluid mechanics is the approximation in the limit $\ep\to0$ of vanishing viscosity (also called {\it inviscid limit} or {\it zero-viscosity limit}) of solution of the Euler equations by solution of the incompressible Navier-Stokes equations, which is naturally associated with the physical phenomena of turbulence
 and of boundary layers. The vanishing viscosity problem is closely related to uniqueness of solution to the
Euler equations, because the methods used to prove uniqueness can often be applied
to show vanishing viscosity. One of the most important uniqueness results in $\R^2$ is due to Yudovich \cite{Yu} and the vanishing viscosity results is due to Chemin \cite{chem}. As mentioned in Kelliher \cite{K}, given the uniqueness of solution to the Euler equations, the compactness argument would imply the vanishing viscosity limit. In the presence of boundaries, the zero-viscosity limit problem is more complicated and essentially difficult as it is closely related to the validity of the Prandtl equation for the formation of boundary layers. We refer to see the recent results by Maekawa \cite{14cpam}, Constantin, Kukavica and Vicol \cite{CKV,CV}, G\'{e}rard et al. \cite{18jmpa} and the references therein. For initial data that is analytic only close to the boundary of the domain, Kukavica, Vicol and Wang \cite{Kuk1,Kuk} proved that the solution of the Navier-Stokes equations converges in the vanishing viscosity limit to the solution of the Euler equations.

Most of the progress has been made in the two-dimensional case without boundary (assuming the domain to be the whole plane $\R^2$ of the torus $\T^2$) since both vortex stretching and boundary effects are absent. In particular, the problem of the convergence of smooth viscous solution of the Navier-Stokes equations to the Eulerian one as $\ep\to0$ is
well understood and has been studied in many literatures, see for example \cite{C17,Swann, Kato}. Majda \cite{Majda} showed that under the assumption $u_0\in H^s$ with $s>\frac{d}{2}+2$, the solution to \eqref{1} converges in $L^2$-norm as $\ep\to 0$ to the unique solution to \eqref{2} and the convergence rate is of order $(\ep t)^{\fr12}$. Masmoudi \cite{M} improved the result and proved the convergence in $H^s$-norm under the weaker assumption $u_0\in H^s$ with $s>\frac{d}{2}+1$. Next we shall recall the consideration of Yudovich solution. Yudovich \cite{Yu} established the uniqueness of solution to the Euler equations in the space $C\left(\mathbb{R}^{+}; L^2(\mathbb{R}^2)\right) \times L_{\text {loc }}^{\infty}\left(\mathbb{R}^{+} ; L^2(\mathbb{R}^2)\right)$ assuming that initial data lies to $L^2(\mathbb{R}^2)$ and initial vorticity belongs to $L^p(\mathbb{R}^2) \cap L^{\infty}(\mathbb{R}^2)$ for some $p<\infty$. For this uniqueness class, Chemin \cite{chem} showed that solution of the Navier-Stokes equations with initial vorticity liying in $L^2(\R^2)\cap L^\infty(\R^2)$ converges in the zero-viscosity limit to a solution of the Euler equations. Kelliher \cite{K} resolved the same inviscid limit as Chemin \cite{chem} with the assumption that the vorticity is unbounded. We also mention that, in the case of that initial velocity has finite energy and initial vorticity is bounded in the plane, Cozzi \cite{co1,co2} showed that the unique solution of the Navier-Stokes equations converges to the unique
solution of the Euler equations in the $L^\infty$-norm uniformly over finite time as viscosity approaches zero.
In the case of two dimensional (torus or the whole space) and rough initial data, by taking greater advantage of vorticity formulation, more beautiful results were obtained quantitatively (see \cite{Be,co3,Ci,P2,K,lop} and the references therein).

Next, we review some results on the inviscid limit under the framework of Besov spaces. In dimension two, Hmidi and Kerrani \cite{HK} proved that \eqref{1} is globally well-posed in Besov space $B^2_{2,1}$, with uniform bounds on the viscosity and obtained that the convergence rate of the inviscid limit is of order $\ep t$ for
vanishing viscosity. Subsequently, in \cite{HK1}, they further generalized to other Besov spaces $B^{2/p+1}_{p,1}$ with convergence in $L^p$. For more details, we also refer to see \cite[Section 3.4]{MWZ}. In dimension three, a similar result was obtained in \cite{Wu} for axis-symmetric flows without swirl.
Bourgain and Li \cite{B1,B2} employed a combination of Lagrangian and Eulerian techniques to obtain strong local ill-posedness results in borderline Besov spaces $B^{d/p+1}_{p,r}$ for $1\leq p<\infty$ and $1<r\leq\infty$ when $d=2,3$. Guo-Li-Yin \cite{GZ} solved the inviscid limit in the same topology. Precisely speaking, they obtained
\begin{theorem}[\cite{GZ}]\label{th1} Let $d\geq 2$ and $\ep\in [0,1]$. Assume that $(s,p,r)$ satisfies
\begin{align*}
&s>\frac{d}{p}+1,\; (p,r)\in [1,\infty]\times [1,\infty)
 \quad \mbox{   or   } \\
&s=\frac{d}{p}+1, \;(p,r)\in [1,\infty]\times\{1\}.
\end{align*}
Given initial data $u_0\in B^{s}_{p,r}(\R^d)$, there exists $T=T(s,p,r,d)>0$ such that \eqref{1} has a unique solution $u^{\rm NS}_{\ep}(t,u_0)\in C([0,T];B^{s}_{p,r})$. Moreover, there holds\\
{(1) (Uniform bounds):}  there exists $C=C(s,p,r,d)>0$ such that
\begin{align*}
\f\|u^{\rm NS}_{\ep}(t,u_0)\g\|_{L_T^\infty B^{s}_{p,r}}\leq C, \quad \forall\ \ep\in [0,1].
\end{align*}
Moreover, if $u_0\in B^{\gamma}_{p,r}$ with $\gamma>s$, there exists $C_2=C_2(\gamma,s,p,r,d)>0$ such that
\begin{align*}
\f\|u^{\rm NS}_{\ep}(t,u_0)\g\|_{L_T^\infty B^\gamma_{p,r}}\leq  C_2\|u_0\|_{B^\gamma_{p,r}}.
\end{align*}
{(2) (Inviscid limit):} $\lim\limits_{\ep\downarrow 0}\f\|u^{\rm NS}_{\ep}(t,u_0)-u^{\rm E}(t,u_0)\g\|_{L_T^\infty B^{s}_{p,r}}=0,$ namely, for any $u_0\in B^{s}_{p,r}$, $\forall \eta>0$, $\exists\ \ep_0=\ep_0(u_0,\eta,T)$, $\forall \ep\leq\ep_0$, such that $$\f\|u^{\rm NS}_{\ep}(t,u_0)-u^{\rm E}(t,u_0)\g\|_{L_T^\infty B^{s}_{p,r}}\leq \eta.$$
\end{theorem}
However, it should be noted that we left an open problem for the case $r=\infty$ in the inviscid limit of the Navie-Stokes equations in the endpoint Besov spaces $B^{s}_{p,\infty}$. Notice that again, the statement (2) is a consequence of the strong convergence of the solution to \eqref{1} with the given initial data belonging in $B^{s}_{p,r}$ with $1\leq r<\infty$.
An interesting {\bf Question} appears:

Given initial data $u_0\in B^{s}_{p,\infty}$, in the inviscid limit $\ep\to0$, whether or not the $B^{s}_{p,\infty}$-convergence
\begin{align*}
u^{\rm NS}_{\ep}(t,u_0)\rightarrow u^{\rm E}(t,u_0)\quad\text{strongly in}  \quad L_T^\infty B^{s}_{p,\infty}
\end{align*}
can be established?

In this paper, we will give a negative answer: By constructing a particular initial data $u_0\in B^{s}_{p,\infty}$, we find that
\begin{align*}
u^{\rm NS}_{\ep}(t,u_0)\nrightarrow u^{\rm E}(t,u_0)\quad\text{strongly in} \quad L_T^\infty B^{s}_{p,\infty}.
\end{align*}

For any $R>0$, from now on, we denote any bounded subset $U_R$ in $B^{s}_{p,\infty}(\mathbb{R}^d)$ by
$$U_R:=\left\{\phi\in B^{s}_{p,\infty}(\mathbb{R}^d): \|\phi\|_{B^{s}_{p,\infty}(\mathbb{R}^d)}\leq R,\;\mathrm{div\,} \phi=0\right\}$$
and also
 \begin{align*}
&u^{\rm E}(t,u_0)=\text{\rm the solution map of}\, \eqref{1}\, \text{\rm with initial data}\, u_0,\\
&u^{\rm NS}_{\ep}(t,u_0)=\text{\rm the solution map of}\, \eqref{2}\, \text{\rm with initial data}\, u_0.
 \end{align*}
Now we state the main result.
\begin{theorem}\label{th2} Let $d\geq 2$ and $\ep\in [0,1]$. Assume that $(s,p)$ satisfies
\begin{align}\label{con1}
s>\frac{d}{p}+1\quad\text{with}\quad p\in [1,\infty].
\end{align}
Given initial data $u_0\in U_R\subseteq B^{s}_{p,\infty}(\mathbb{R}^d)$, then \\
{(1)} there exists a positive time $T_1$ such that \eqref{2} has a unique solution $u^{\rm E}(t,u_0)\in L^\infty_{T_1}B^{s}_{p,\infty}$;\\
{(2)} for all $\ep>0$, there exists a positive time $T_2\geq T_1$ such that \eqref{1} has a unique solution $u^{\rm NS}_{\ep}(t,u_0) \in L^\infty_{T_2}B^{s}_{p,\infty}$;\\
{(3) \bf (Non-convergence)}: there exists an initial data $u_0\in B^{s}_{p,\infty}(\mathbb{R}^d)$ such that
the solution $u^{\rm NS}_{\ep}(t,u_0)$ of \eqref{1} does not converge to the solution $u^{\rm E}(t,u_0)$ of \eqref{2} for small $t\in(0,T_1]$ in $B^{s}_{p,\infty}(\R^d)$ as $\ep\downarrow 0$.
More precisely, for all $\ep>0$, there exists an initial data $u_0\in B^{s}_{p,\infty}(\mathbb{R}^d)$ and a small enough $\delta>0$ such that for a short time $t=\delta \varepsilon$
\begin{align}\label{impor}
\left\|u^{\rm NS}_{\ep}(t,u_0)-u^{\rm E}(t,u_0)\right\|_{B^{s}_{p,\infty}}\geq \eta_0,
\end{align}
with some positive constant $\eta_0$ depending on $d$ and $\delta$ but independent of $\ep$.
\end{theorem}

\begin{remark}\label{re1}
The novelty of Theorem \ref{th2} is part (3), while part (1) and (2) is classical (see \cite[Remark 2]{GZ}).
\end{remark}

\noindent\textbf{Strategies to the proof of Theorem \ref{th2}.}

Due to the incompressibility condition, the pressure can be eliminated from \eqref{1} and \eqref{2}. In fact, applying the Leray operator $\mathcal{P}$ to \eqref{1} and \eqref{2}, respectively, then we have
\begin{align*}
\begin{cases}
\pa_t u-\ep \Delta u+\mathcal{P}(u\cdot \nabla u)=0,\\
u(0,x)=u_0(x),
\end{cases}
\end{align*}
and
\begin{align*}
\begin{cases}
\pa_t u+\mathcal{P}(u\cdot \nabla u)=0, \\
u(0,x)=u_0(x).
\end{cases}
\end{align*}
From Duhamel's principle, it follows that
\begin{align}
&u^{\rm NS}_{\ep}(t,u_0)=e^{t\ep\Delta}u_0-\int_0^te^{(t-\tau)\ep\Delta}\mathcal{P}\f(u^{\rm NS}_{\ep}\cdot \nabla u^{\rm NS}_{\ep}\g)\dd\tau,\label{yh1}\\
&u^{\rm E}(t,u_0)=u_0-\int_0^t\mathcal{P}\f(u^{\rm E}\cdot \nabla u^{\rm E}\g)\dd\tau.\label{yh2}
\end{align}
\noindent{\bf Step 1:}\;At a first sight, to ensure that the difference between the solution $u^{\rm NS}_{\ep}$ to \eqref{1} and $u^{\rm E}$ to \eqref{2} in the $B^s_{p,\infty}$-topology is bounded below by a positive constant at any later time, we need to construct special initial data $u_0\in B^s_{p,\infty}$ to satisfy that the difference between $e^{t\ep\Delta}u_0$ and $u_0$ in the $B^s_{p,\infty}$-topology is bounded below by a positive constant at any later time. In order to see this, let us introduce the following toy-model:

\noindent{\bf A toy-model: 1-D heat equation.}
Assume that $(s,p)$ satisfies the condition \eqref{con1} with $d=1$. Let $\ep_n:=2^{-2n}$ with $n\gg1$. The Cauchy problem
\begin{align*}
\begin{cases}
\pa_t u-\ep_n \pa^2_{xx} u=0,\\
u(0,x)=u_0(x)\in B^s_{p,\infty}(\R),
\end{cases}
\end{align*}
has a unique explicit solution $u^{\ep_n}(x,t)=e^{\ep_nt\pa^2_{xx} }u_0(x)$, where the heat kernel operator is given by $e^{\ep_nt\pa^2_{xx}}:=\mathcal{F}^{-1}e^{-\ep_nt|\xi|^2}\mathcal{F}.$

Let $\phi$ be given in subsection \ref{sec3.1} and set $u_0(x)=2^{-ns}\phi(x)\cos\f(\frac{17}{12}2^nx\g)$, then $u^{\ep_n}-u_0\nrightarrow 0$ in $B^s_{p,\infty}(\R)$.
In fact, for $n\gg1$,
notice that the frequency of $u_0$ is supported in the annulus $\{\xi\in \R:\ \frac{33}{24}2^{n}\leq |\xi|\leq \frac{35}{24}2^{n}\}$, one has for small $t>0$ and for some positive constants $c_1,C_1,C_2$ independent of $t$ and $n$
\begin{align*}
C_2t\geq&~\f\|e^{\ep_nt\pa^2_{xx} }u_0-u_0\g\|_{B^s_{p,\infty}}=2^{ns}\big\|\De_{n}\big(e^{t\ep_n\pa^2_{xx}}u_0-u_0\big)\big\|_{L^p}\nonumber\\
\geq&~  t2^{ns}\big\|{\De}_{n}\ep_n\pa^2_{xx}u_0\big\|_{L^p}-\int_0^t2^{ns}\left\|{\De}_{n}\ep_n\pa^2_{xx}\big(e^{\tau\ep_n\pa^2_{xx}}u_0-u_0 \big)\right\|_{L^{p}}\dd\tau\nonumber\\
\geq&~ ct\f\|\phi(x)\cos\f(\frac{17}{12}2^nx\g)\g\|_{L^p}-C\int_0^t2^{ns}\left\|{\De}_{n}\big(e^{\tau\ep_n\pa^2_{xx}}u_0-u_0 \big)\right\|_{L^{p}}\dd\tau\nonumber\\
\geq&~ c_1t-C_1t^2.
 \end{align*}
For more details see subsection \ref{sec3.1}.

\noindent{\bf Step 2:}\;It remains to prove that the error term
$$\int_0^te^{(t-\tau)\ep\Delta}\mathcal{P}(u^{\rm NS}_{\ep}\cdot \nabla u^{\rm NS}_{\ep})\dd\tau-\int_0^t\mathcal{P}\f(u^{\rm E}\cdot \nabla u^{\rm E}\g)\dd\tau$$
is smaller than the leading term $e^{t\ep\Delta}u_0-u_0$.
We can decompose the above error term into four terms (see \eqref{fj-1} in subsection \ref{sec3.4}) and estimate them by some basic analytical techniques (see Propositions \ref{pr2}-\ref{pr3} in subsection \ref{sec3.3}).

\noindent\textbf{Organization of our paper.}

In Section \ref{sec2}, we list some notations and known results which will be used in the sequel. We prove Theorem \ref{th2} in Section \ref{sec3} by dividing it into several parts: \\
(1) Construction and Estimation of initial data (see subsection \ref{sec3.1});\\
(2) Estimation of leading term (see subsection \ref{sec3.2});\\
(3) Estimation of error terms (see subsection \ref{sec3.3});\\
(4) Non-convergence (see subsection \ref{sec3.4}).
\section{Preliminaries}\label{sec2}
We will use the following notations throughout this paper.
\begin{itemize}
  \item For $X$ a Banach space and $I\subset\R$, we denote by $\mathcal{C}(I;X)$ the set of continuous functions on $I$ with values in $X$. Sometimes we will denote $L^p(0,T;X)$ by $L_T^pX$.
  \item The symbol $\mathrm{A}\approx \mathrm{B}$ means that there is a uniform positive ``harmless" constant $C$ independent of $\mathrm{A}$ and $\mathrm{B}$ such that $C^{-1}\mathrm{B}\leq \mathrm{A}\leq C\mathrm{B}$.
  \item For all $u\in \mathcal{S}(\mathbb{R}^{d})$ (Schwartz space), the Fourier transform $\mathcal{F}u$ (also denoted by $\widehat{u}$) is defined by
$$
\f(\mathcal{F}u\g)(\xi)=\widehat{u}(\xi)=\int_{\mathbb{R}^{d}}e^{-\mathrm{i}x\cdot\xi}u(x)\dd x \quad\text{for any}\; \xi\in\mathbb{R}^{d}.
$$
\item For all $u\in \mathcal{S}(\mathbb{R}^{d})$, the inverse Fourier transform is defined by $(\mathcal{F}^{-1}u)=(2\pi)^{-d}\f(\mathcal{F}u\g)(-\xi)$, namely,
$$
(\mathcal{F}^{-1}u)(x)=(2\pi)^{-d}\int_{\R^d}e^{\mathrm{i}x\cdot\xi}u(\xi)\dd\xi.
$$
  \item For all $F\in \mathcal{S'}(\mathbb{R}^{d})$ (space of tempered distributions), the Fourier transform $\mathcal{F}F=\widehat{F}$ is defined by
$$
\langle\mathcal{F}F,u\rangle=\langle F,\mathcal{F}u\rangle\quad\text{for all}\; u\in\mathcal{S}(\R^d).
$$
 \item We denote the projection
\bbal
&\mathcal{P}: L^{p}(\mathbb{R}^{d}) \rightarrow L_{\sigma}^{p}(\mathbb{R}^{d}) \equiv \overline{\left\{f \in \mathcal{C}^\infty_{0}(\mathbb{R}^{d}) ; {\rm{div}} f=0\right\}}^{\|\cdot\|_{L^{p}(\mathbb{R}^{d})}},\quad p\in(1,\infty),\\
&\mathcal{Q}=\mathrm{Id}-\mathcal{P}.
\end{align*}
In $\mathbb{R}^{d}$, $\mathcal{P}$ can be defined by $\mathcal{P}= \mathrm{Id}+(-\Delta)^{-1}\nabla {\rm{div}}$, or equivalently, $\mathcal{P}=(\mathcal{P}_{i j})_{1 \leqslant i, j \leqslant d}$, where $\mathcal{P}_{i j} \equiv \delta_{i,j}+R_{i} R_{j}$ with $\delta_{i,j}$ being the Kronecker delta ($\delta_{i,j}=0$ for $i\neq j$ and $\delta_{i,i}=0$) and $R_{i}$ being the Riesz transform with symbol $-\mathrm{i}\xi_1/|\xi|$.
\end{itemize}
Next, we will recall some facts about the Littlewood-Paley decomposition and the nonhomogeneous Besov spaces (see \cite{B.C.D} for more details).
Choose a radial, non-negative, smooth function $\vartheta:\R^d\mapsto [0,1]$ such that
\begin{itemize}
  \item ${\rm{supp}} \;\vartheta\subset B(0, 4/3)$;
  \item $\vartheta(\xi)\equiv1$ for $|\xi|\leq3/4$.
\end{itemize}
Setting $\varphi(\xi):=\vartheta(\xi/2)-\vartheta(\xi)$, then we deduce that $\varphi$ has the following properties
\begin{itemize}
  \item ${\rm{supp}} \;\varphi\subset \left\{\xi\in \R^d: 3/4\leq|\xi|\leq8/3\right\}$;
  \item $\varphi(\xi)\equiv 1$ for $4/3\leq |\xi|\leq 3/2$;
  \item $\vartheta(\xi)+\sum_{j\geq0}\varphi(2^{-j}\xi)=1$ for any $\xi\in \R^d$;
  \item $\sum_{j\in \mathbb{Z}}\varphi(2^{-j}\xi)=1$ for any $\xi\in \R^d\setminus\{0\}$.
\end{itemize}
For any $u\in \mathcal{S'}(\R^d)$, the nonhomogeneous and homogeneous dyadic blocks are defined as follows
\begin{align*}
&\Delta_ju=0,\; \text{if}\; j\leq-2;\\
&\Delta_{-1}u=\vartheta(D)u;\\
&\Delta_ju=\varphi(2^{-j}D)u,\; \; \text{if}\;j\geq0,
\end{align*}
and for any $u\in \mathcal{S}'_h(\R^d)$
\begin{align*}
\dot{\Delta}_ju=\varphi(2^{-j}D)u,\; \; \text{if}\;j\in \mathbb{Z},
\end{align*}
where the pseudo-differential operator is defined by $\sigma(D):u\to\mathcal{F}^{-1}(\sigma \mathcal{F}u)$ and $\mathcal{S}'_h$ is given by
\begin{eqnarray*}
\mathcal{S}'_h:=\Big\{u \in \mathcal{S'}(\mathbb{R}^{d}):\; \lim_{j\rightarrow-\infty}\|\vartheta(2^{-j}D)u\|_{L^{\infty}}=0 \Big\}.
\end{eqnarray*}

\begin{remark}\label{not}
Assume that $\mathrm{supp} \ \widehat{f_j}\subset \left\{\xi\in\R^d: \ \frac{33}{24}2^{j}\leq |\xi|\leq \frac{35}{24}2^{j}\right\}$. Obviously, $\varphi(2^{-j}\xi)=1$ if $\xi\in \mathrm{supp} \ \widehat{f_j}$. Also, $\mathrm{supp} \ \widehat{f_j}\,\cap\, \mathrm{supp} \ \varphi(2^{-n}\cdot)=\emptyset$ if $n\neq j$.
Thus
we have for $-1\leq j,n\in \mathbb{Z}$
\begin{align*}
{\Delta_n(f_j)=\mathcal{F}^{-1}\f(\varphi(2^{-n}\cdot)\widehat{f}_{j}\g)=}
\begin{cases}
f_j, &\text{if}\; n=j,\\
0, &\text{if}\; n\neq j.
\end{cases}
\end{align*}
\end{remark}
We recall the definition of the Besov spaces and norms.
\begin{definition}[Besov Spaces, \cite{B.C.D}]
Let $s\in\mathbb{R}$ and $(p,r)\in[1, \infty]^2$. The nonhomogeneous Besov spaces are defined as follows
\begin{align*}
B_{p,r}^{s}(\R^d)=\f\{f \in \mathcal{S}'(\mathbb{R}^{d}):\; \|f\|_{B_{p,r}^{s}(\mathbb{R}^{d})}< \infty \g\},\quad\text{where}
\end{align*}
\begin{numcases}{\|f\|_{B^{s}_{p,r}(\R^d)}:=}
\left(\sum_{j\geq-1}2^{sjr}\|\Delta_jf\|^r_{L^p(\R^d)}\right)^{1/r}, &if $1\leq r<\infty$,\nonumber\\
\sup_{j\geq-1}\f(2^{sj}\|\Delta_jf\|_{L^p(\R^d)}\g), &if $r=\infty$.\nonumber
\end{numcases}
The homogeneous Besov spaces are defined as follows
\begin{align*}
\dot{B}_{p,r}^{s}(\R^d)=\Big\{f \in \mathcal{S}'_h(\mathbb{R}^{d}):\; \|f\|_{\dot{B}_{p,r}^{s}(\mathbb{R}^{d})}< \infty \Big\},\quad\text{where}
\end{align*}
\begin{numcases}{\|f\|_{\dot{B}^{s}_{p,r}(\R^d)}:=}
\left(\sum_{j\in\mathbb{Z}}2^{sjr}\|\dot{\Delta}_jf\|^r_{L^p(\R^d)}\right)^{1/r}, &if $1\leq r<\infty$,\nonumber\\
\sup_{j\in\mathbb{Z}}\f(2^{sj}\|\dot{\Delta}_jf\|_{L^p(\R^d)}\g), &if $r=\infty$.\nonumber
\end{numcases}
\end{definition}
We remark that, for any $s>0$ and $(p,r)\in[1, \infty]^2$, then $B^{s}_{p,r}(\R^d)=\dot{B}^{s}_{p,r}(\R^d)\cap L^p(\R^d)$ and
$$\|f\|_{B^{s}_{p,r}}\approx \|f\|_{L^{p}}+\|f\|_{\dot{B}^{s}_{p,r}}.$$

The following Bernstein's inequalities will be used in the sequel.
\begin{lemma}[Bernstein's inequalities, \cite{B.C.D}] \label{lem2.1} Let $\mathcal{B}$ be a ball and $\mathcal{C}$ be an annulus. There exists a positive constant $C$ such that for all $k\in \mathbb{N}\cup \{0\}$, any $\lambda\in \R^+$ and any function $f\in L^p$ with $1\leq p \leq q \leq \infty$
\begin{align*}
&{\rm{supp}}\ \widehat{f}\subset \lambda \mathcal{B}\;\Rightarrow\; \|D^kf\|_{L^q}\leq C^{k+1}\lambda^{k+(\frac{d}{p}-\frac{d}{q})}\|f\|_{L^p},  \\
&{\rm{supp}}\ \widehat{f}\subset \lambda \mathcal{C}\;\Rightarrow\; C^{-k-1}\lambda^k\|f\|_{L^p} \leq \|D^kf\|_{L^p} \leq C^{k+1}\lambda^k\|f\|_{L^p}.
\end{align*}
\end{lemma}
Finally we recall the following product law which will be used often in the sequel.
\begin{lemma}[Product estimation, \cite{B.C.D}]\label{cj} Assume $(s,p)$ satisfies \eqref{con1}.
$B^{s}_{p,\infty}(\R^d)$ is a Banach algebra.
For any $r\in[1,\infty]$, then
 there exists a constant $C>0$ such that
$$\|fg\|_{B^{s}_{p,r}}\leq C\f(\|f\|_{L^\infty}\|g\|_{B^s_{p,r}}+\|g\|_{L^\infty}\|f\|_{B^{s}_{p,r}}\g),\quad \forall f,g\in L^\infty\cap B^{s}_{p,r}.$$
Furthermore, due to the embedding $B^{s-1}_{p,r}(\R^d)\hookrightarrow L^{\infty}(\R^d)$, there holds
$$\|f\cdot \na g\|_{B^{s-1}_{p,r}}\leq C\|f\|_{B^{s-1}_{p,r}}\|g\|_{B^s_{p,r}},\quad \forall(f,g)\in B^{s-1}_{p,r}\times B^s_{p,r}.$$
\end{lemma}

\section{Proof of Theorem \ref{th2}}\label{sec3}
In this section, we prove Theorem \ref{th2} by dividing it into several parts:
(1) Construction and Estimation of initial data;
(2) Estimation of leading term;
(3) Estimation of errors terms;
(4) Non-convergence.

\subsection{Construction of Initial Data}\label{sec3.1}
We need to introduce smooth, radial cut-off functions to localize the frequency region. Let $\widehat{\phi}\in \mathcal{C}^\infty_0(\mathbb{R})$ be an even, real-valued and non-negative function on $\R$ that satisfies
\begin{align*}{\widehat{\phi}(\xi)=}
\begin{cases}
1, & \mathrm{if} \; |\xi|\leq \frac{1}{4^d},\\
0, & \mathrm{if} \; |\xi|\geq \frac{1}{2^d}.\end{cases}
\end{align*}
We assume that $\widehat{\phi}\in \mathcal{C}^\infty_0(\mathbb{R})$ is even and real-valued such that $\phi(x)=(2\pi)^{-1}\mathcal{F}^{-1}(\widehat{\phi}(\xi))$ is a real-valued and continuous function on $\R$. It is easy to check that $\|\phi\|_{L^\infty}=\phi(0)>0$. Then there exists some $r>0$ such that
$\phi(x)\geq \phi(0)/{2}$ for any $x\in B_{r}(0).$
Moreover, we have
\begin{remark}\label{ley2}
For any $p\in[1,\infty]$, there exists two positive constants $c_1$ and $C_1$, such that
\begin{align*}
c_1\leq\|\phi\|_{L^p(\R)}\leq C_1.
\end{align*}
\end{remark}

\begin{remark}\label{ley54}
Let $f(x)\in\{\sin x,\cos x\}$. For any $p\in[1,\infty]$, there exists some $r>0$ and positive constant $M$, such that for $n\gg1$
\bbal
\left\|\phi(x)f\left(\frac{17}{12}2^{n}x\right)\right\|_{L^p(\R)}\geq \frac{\phi(0)}{2}\left\|f\left(\frac{17}{12}2^{n}x\right)\right\|_{L^p(B_{r}(0))}\geq M.
\end{align*}
In fact, for the case when $p=\infty$, we have for some $n$ large enough
\bbal
\left\|f\left(\frac{17}{12}2^{n}x\right)\right\|_{L^{\infty}(B_{r}(0))}=\left\|f(x)\right\|_{L^{\infty}(B_{\lambda_n}(0))}=1\quad\text{with}\quad \lambda_n:=\fr{17}{12}r2^{n}.
\end{align*}
For the case when $p\in[1,\infty)$, there exists a positive integer number $N$ such that for $n>N$
\bbal
\left\|f\left(\frac{17}{12}2^{n}x\right)\right\|^p_{L^p(B_{r}(0))}&=\frac{2r}{\lambda_n}\int^{\lambda_n}_{0}|f(x)|^p\dd x\geq\frac{r}{\pi}\int^\pi_0|f(x)|^p\dd x,
\end{align*}
since
\bbal
\lim_{n\rightarrow \infty}\frac{1}{\lambda_n}\int_0^{\lambda_n}|f(x)|^p\dd x=\frac{1}{\pi}\int^\pi_0|f(x)|^p\dd x.
\end{align*}
\end{remark}

\begin{lemma}\label{ley3} Let $s\in \R$ and $p\in[1,\infty]$. For some fixed $\delta\in(0,1)$ which will be determined later, we define the divergence-free vector field $u_0$  by
\begin{align*}
&u_0(x):=\delta\f(
-\pa_2,\;
\pa_1,\;
0,\;
\ldots,\;
0
\g)\sum_{j=3}^{\infty}2^{-j(s+1)}f_j(x)
,\quad \text{where}\\
&f_j(x):=\phi(x_1)\cos \left(\frac{17}{12}2^jx_1\right)\prod_{i=2}^d\phi(x_i).
\end{align*}
Then there exists some sufficiently large $n\in \mathbb{Z}^+$ such that
\begin{align}\label{z3}
&\|u_0\|_{B^s_{p,\infty}}\thickapprox \delta\quad\text{and}\quad 2^{ns}\|\De_{n}u_0\|_{L^p}\thickapprox \delta.
\end{align}
\end{lemma}
\begin{proof}\,
By easy computations, we deduce that
\begin{align*}
\widehat{f_j}(\xi)
=\fr12 \left[\widehat{\phi}\left(\xi_1+\frac{17}{12} 2^{j}\right)+\widehat{\phi}\left(\xi_1-\frac{17}{12} 2^{j}\right)\right]\prod_{i=2}^d\widehat{\phi}\f(\xi_i\g).
\end{align*}
Due to the support condition of $\widehat{\phi}$, we see that
\bbal
&\mathrm{supp} \ \widehat{f_j}(\xi)\subset \mathcal{S}_1 \cup \mathcal{S}_2,
\end{align*}
where
\bbal
&\mathcal{S}_1:= \left\{\xi\in\R^d: \left|\xi_1+\frac{17}{12} 2^{j}\right|\leq \frac{1}{2^{d}},\ |\xi_i|\leq \frac{1}{2^{d}},\ 2\leq i\leq d\right\}, \\
&\mathcal{S}_2:= \left\{\xi\in\R^d: \left|\xi_1-\frac{17}{12} 2^{j}\right|\leq \frac{1}{2^{d}},\ |\xi_i|\leq \frac{1}{2^{d}},\ 2\leq i\leq d\right\},
\end{align*}
which implies that
 \begin{align*}
\mathrm{supp} \ \widehat{f_j}
&\subset \left\{\xi\in\R^d: \ \frac{33}{24}2^{j}\leq |\xi|\leq \frac{35}{24}2^{j}\right\}.
\end{align*}
Noticing that $\Delta_k(f_j)=\delta_{k,j}f_j$ for some $k\geq-1$ (see Remark \ref{not}),
we deduce that
 \begin{align*}
 \Delta_{k}u_0:=\delta2^{-k(s+1)}
\f(
-\pa_2f_k,\;
\pa_1f_k,\;
0,\;
\cdots,\;
0
\g),
\end{align*}
which implies that
\begin{align*}
 \|u_{0}\|_{{B}_{p,\infty}^{s}(\R^d)}&= \sup_{k\geq -1}\f(2^{ks}\|\Delta_{k}u_{0}\|_{L^{p}(\R^d)}\g)\\
 &= \delta\sup_{k\geq 0}\f(2^{-k}\|\pa_1f_k\|_{L^{p}(\R^d)}+2^{-k}\|\pa_2f_k\|_{L^{p}(\R^d)}\g)\thickapprox \delta,
\end{align*}
and
\begin{align*}
 2^{ns}\|\Delta_{n}u_0\|_{L^p(\R^d)}&
= \delta2^{-n}\f(\|\pa_1f_n\|_{L^{p}(\R^d)}+\|\pa_2f_n\|_{L^{p}(\R^d)}\g)\thickapprox \delta,
\end{align*}
where we have used Remark \ref{ley2} and Remark \ref{ley54}.
We complete the proof of Lemma \ref{ley2}.
\end{proof}
\subsection{Estimation of Leading Term}\label{sec3.2}
Now we establish the following proposition involving the leading term $e^{t\ep\Delta}u_0-u_0$ which will affect the non-convergence of solutions to \eqref{1} in the proof of Theorem \ref{th2}.

From now on, we set $\ep:=2^{-n}$ with $n\gg 1$.
\begin{proposition}\label{pr0}\label{lmm1} Let $u_0$ be given by Lemma \ref{ley3}. Assume that $(s,p)$ satisfies \eqref{con1}. For small $t>0$, we have for some positive constants $c_1,C_1,C_2$ independent of $t,\ep$ and $\delta$
\begin{align*}
\delta t2^n\f(c_1-C_1t2^{n}\g)\leq 2^{ns}\left\|\De_{n}\big(e^{t\ep\Delta}u_0-u_0\big)\right\|_{L^p}\leq  C_2\delta 2^nt.
\end{align*}
\end{proposition}
\begin{proof}  By the Newton-Leibniz formula, one has
\bal\label{ml}
e^{t\ep\Delta}u_0-u_0=\int_0^t\f(e^{\tau\ep \Delta} \ep\Delta u_0\g) \dd\tau.
\end{align}
From \eqref{ml}
and using the classical $L^p$-$L^p$ estimate: $\|e^{\ep \tau\Delta}f\|_{L^p} \leq \|f\|_{L^p}$ with $1\leq p\leq \infty$, we obtain that for large $n$
\bbal
2^{ns}\big\|\De_{n}\big(e^{t\ep\Delta}u_0-u_0\big)\big\|_{L^p}
&\leq  t2^{ns}\big\|{\De}_{n}\f(\ep\Delta u_0\g)\big\|_{L^p}\leq  Ct\delta2^n.
\end{align*}
In fact, for large $n$, by Bernstein's inequality (see Lemma \ref{lem2.1}) and from \eqref{z3}, we have
$$2^{ns}\big\|{\De}_{n}\f(\ep\Delta u_0\g)\big\|_{L^p}\thickapprox 2^{-n}2^{2n}2^{ns}\big\|{\De}_{n} u_0\big\|_{L^p}\thickapprox  \delta2^n.$$
Combining the above, we obtain that for large $n$
\bbal
2^{ns}\big\|\De_{n}\big(e^{t\ep\Delta}u_0-u_0\big)\big\|_{L^p}
\geq&~  t2^{ns}\big\|{\De}_{n}\ep\Delta u_0\big\|_{L^p}-2^{ns}\int_0^t\left\|{\De}_{n}\ep\Delta\big(e^{\tau\ep\Delta}u_0-u_0 \big)\right\|_{L^{p}}\dd\tau\nonumber\\
\geq&~ ct2^{ns}2^n\big\|{\De}_{n}u_0\big\|_{L^p}-C2^n2^{ns}\int_0^t\left\|{\De}_{n}\big(e^{\tau\ep\Delta}u_0-u_0 \big)\right\|_{L^{p}}\dd\tau\nonumber\\
\geq&~ c_1t\delta2^n-C_1t^2\delta 2^{2n}\nonumber\\
\geq&~\delta t2^n\f(c_1-C_1t2^{n}\g).
\end{align*}
Thus we complete the proof of Proposition \ref{lmm1}.
\end{proof}
\subsection{Estimation of Error Terms}\label{sec3.3}
The following propositions involving the perturbation of solutions will play a crucial role in the proof of Theorem \ref{th2}.
\begin{proposition}\label{pr1} Let $u_0$ be given by Lemma \ref{ley3}. Assume that $(s,p)$ satisfies \eqref{con1}. For $t\in (0,1)$, we have for some positive constant $C$ independent of $t,\ep$ and $\delta$
\begin{align*}
&\left\|e^{t\ep\Delta}u_0\cd\na e^{t\ep\Delta}u_0-u_0\cd\na u_0\right\|_{\dot{B}^{s-1}_{p,\infty}}\leq C\delta^2,\\
&\left\|\left(e^{t\ep\Delta}-\mathrm{Id}\right)\left(u_0\cd\na u_0\right)\right\|_{\dot{B}^{s-1}_{p,\infty}}\leq C\delta^2.
\end{align*}
\end{proposition}
\begin{proof} Noticing that $\mathrm{div\,} u_0=0$ and using Lemmas \ref{cj}-\ref{lmm1}, we have
\bbal
\left\|e^{t\ep\Delta}u_0\cd\na e^{t\ep\Delta}u_0-u_0\cd\na u_0\right\|_{\dot{B}^{s-1}_{p,\infty}}
\leq&~ C\left\|e^{t\ep\Delta}u_0\otimes e^{t\ep\Delta}u_0-u_0\otimes u_0\right\|_{B^{s}_{p,\infty}}\\
\leq&~ C\left\|u_0\right\|^2_{B^{s}_{p,\infty}}\\
\leq&~ C\delta^2,
\end{align*}
and
\bbal
\left\|\left(e^{t\ep\Delta}-\mathrm{Id}\right)\left(u_0\cd\na u_0\right)\right\|_{\dot{B}^{s-1}_{p,\infty}}
\leq&~ C\left\|u_0\otimes u_0\right\|_{B^{s}_{p,\infty}}
\leq C\delta^2.
\end{align*}
This finishes the proof of Proposition \ref{pr1}.
\end{proof}

\begin{proposition}\label{pr2}
Let $u_0$ be given by Lemma \ref{ley3}. Assume that $(s,p)$ satisfies \eqref{con1}. For small $t>0$, we have for some positive constant $C$ independent of $t,\ep$ and $\delta$
\bbal
\left\|u^{\rm E}(t,u_0)-u_0+t\mathcal{P}\left(u_0\cd\na u_0\right)\right\|_{\dot{B}^{s-2}_{p,\infty}}\leq Ct^{2}\delta^3.
\end{align*}
\end{proposition}
\begin{proof} From now on, we set $u^{\rm E}_t=u^{\rm E}(t,u_0)$ for simplicity. By \cite[Theorem 7.1]{B.C.D}, we know that there exists a positive time $T_1=T_1(\|u_0\|_{B^s_{p,\infty}})$ such that \eqref{2} has a unique solution $u^{\rm E}_t\in L^\infty_{T_1}B^s_{p,\infty}$ which satisfies
\bbal
\|u^{\rm E}_t\|_{L^\infty_{T_1}B^s_{p,\infty}}\leq C\|u_0\|_{B^s_{p,\infty}}\leq C\delta.
\end{align*}
By the Fundamental Theorem of Calculus, from \eqref{2}, one has
\bal\label{y}
u^{\rm E}_t-u_0=-\int_0^t\mathcal{P}\left(u^{\rm E}_\tau\cd\na u^{\rm E}_\tau\right)\dd\tau.
\end{align}
By Minkowski's inequality and Lemma \ref{cj}, we deduce that
\bal\label{z8}
\left\|u^{\rm E}_t-u_0\right\|_{B^{s-1}_{p,\infty}}
\leq&~ \int_0^t\left\|\mathcal{P}\left(u^{\rm E}_\tau\cd\na u^{\rm E}_\tau\right)\right\|_{B^{s-1}_{p,\infty}}\dd\tau\nonumber\\
\leq&~ \int_0^t\left\|\mathcal{P}\left(u^{\rm E}_\tau\cd\na u^{\rm E}_\tau\right)\right\|_{\dot{B}^{s-1}_{p,\infty}\cap\dot{B}^0_{p,1}} \dd\tau\nonumber\\
\leq&~ C\int_0^t\left\|u^{\rm E}_\tau\otimes u^{\rm E}_\tau\right\|_{\dot{B}^{s}_{p,\infty}\cap\dot{B}^1_{p,1}} \dd\tau\nonumber\\
\leq&~ C\int_0^t\left\|u^{\rm E}_\tau\otimes u^{\rm E}_\tau\right\|_{B^{s}_{p,\infty}} \dd\tau\nonumber\\
\leq&~ Ct\|u_0\|_{B^{s}_{p,\infty}}^2\nonumber\\
\leq&~ Ct\delta^2,
\end{align}
where we have used the boundedness of Riesz transforms in homogeneous Besov spaces and the fact that $B^s_{p,\infty}$ is a Banach algebra.

From \eqref{y}, it follows that
\bbal
u^{\rm E}_t-u_0+t\mathcal{P}\left(u_0\cd\na u_0\right)=\int_0^t\mathcal{P}\left(u_0\cd\na u_0-u^{\rm E}_\tau\cd\na u^{\rm E}_\tau\right)\dd\tau.
\end{align*}
Thus, we obtain from \eqref{z8} that
\bbal
\left\|u^{\rm E}_t-u_0+t\mathcal{P}\left(u_0\cd\na u_0\right)\right\|_{\dot{B}^{s-2}_{p,\infty}}
\leq&~ C\int_0^t\left\|u_0\cd\na u_0-u^{\rm E}_\tau\cd\na u^{\rm E}_\tau\right\|_{\dot{B}^{s-2}_{p,\infty}}\dd\tau\\
\leq&~ C\int_0^t\left\|u_0\otimes u_0-u^{\rm E}_\tau\otimes u^{\rm E}_\tau\right\|_{B^{s-1}_{p,\infty}}\dd\tau\\
\leq&~ C\int_0^t\left\|u^{\rm E}_\tau-u_0\right\|_{B^{s-1}_{p,\infty}}\f(\|u_0\|_{B^{s}_{p,\infty}}+\|u^{\rm E}_\tau\|_{B^{s}_{p,\infty}}\g)\dd\tau\\
\leq&~ Ct^2\delta^3.
\end{align*}
Then we complete the proof of Proposition \ref{pr2}.
\end{proof}
\begin{proposition}\label{pr3}
Let $u_0$ be given by Lemma \ref{ley3}. Assume that $(s,p)$ satisfies \eqref{con1}. For small $t>0$, we have for some positive constant $C$ independent of $t,\ep$ and $\delta$
\bbal
\left\|u^{\rm NS}_{\ep}(t,u_0)-e^{t\ep\Delta}u_0+\int_0^te^{(t-\tau)\ep\Delta}\mathcal{P}\left(e^{\tau\ep\Delta}u_0\cd\na e^{\tau\ep\Delta}u_0\right)\dd\tau\right\|_{\dot{B}^{s-2}_{p,\infty}}\leq Ct^{2}\delta^3.
\end{align*}
\end{proposition}
\begin{proof} From now on, we set $u^{\rm NS}_t=u^{\rm NS}_{\varepsilon}(t,u_0)$ for simplicity.  By \cite[Remark 2]{GZ}, we know  that there exists a positive time $T_2\geq T_1$ such that
\bbal
\|u^{\rm NS}_t\|_{L^\infty_{T_2}B^s_{p,\infty}}\leq C\|u_0\|_{B^s_{p,\infty}}\leq C\delta.
\end{align*}
By Duhamel's principle, from \eqref{1}, one has
\bal\label{l}
u^{\rm NS}_t=e^{t\ep\Delta}u_0-\int_0^te^{(t-\tau)\ep\Delta}\mathcal{P}\left(u^{\rm NS}_\tau\cd\na u^{\rm NS}_\tau\right)\dd\tau,
\end{align}
which gives that
\bal\label{LL}
\left\|u^{\rm NS}_t-e^{t\ep\Delta}u_0\right\|_{B^{s-1}_{p,\infty}}
\leq&~ C\int_0^t\left\|u^{\rm NS}_\tau\cd\na u^{\rm NS}_\tau\right\|_{B^{s-1}_{p,\infty}}\dd\tau\nonumber\\
\leq&~ C\int_0^t\left\|u^{\rm NS}_\tau\otimes u^{\rm NS}_\tau\right\|_{B^{s}_{p,\infty}}\dd\tau\nonumber\\
\leq&~ Ct\delta^2.
\end{align}
From \eqref{l}, we have
\bbal
u^{\rm NS}_t&-e^{t\ep\Delta}u_0+\int_0^te^{(t-\tau)\ep\Delta}\mathcal{P}\left(e^{\tau\ep\Delta}u_0\cd\na e^{\tau\ep\Delta}u_0\right)\dd\tau\nonumber\\
&=-\int_0^te^{(t-\tau)\ep\Delta}\mathcal{P}\left(u^{\rm NS}_\tau\cd\na u^{\rm NS}_\tau-e^{\tau\ep\Delta}u_0\cd\na e^{\tau\ep\Delta}u_0\right)\dd\tau,
\end{align*}
from which and Minkowski's inequality, and using \eqref{LL}, we obtain that
\bbal
&\left\|u^{\rm NS}_t-e^{t\ep\Delta}u_0+\int_0^te^{(t-\tau)\ep\Delta}\mathcal{P}\left(e^{\tau\ep\Delta}u_0\cd\na e^{\tau\ep\Delta}u_0\right)\dd\tau\right\|_{\dot{B}^{s-2}_{p,\infty}}\\
\leq&~ C\int_0^t\left\|\mathcal{P}\left(u^{\rm NS}_\tau\cd\na u^{\rm NS}_\tau-e^{\tau\ep\Delta}u_0\cd\na e^{\tau\ep\Delta}u_0\right)\right\|_{\dot{B}^{s-2}_{p,\infty}}\dd\tau\\
\leq&~ C\int_0^t\left\|u^{\rm NS}_\tau-e^{\tau\ep\Delta}u_0\right\|_{B^{s-1}_{p,\infty}} \f(\|u_0\|_{B^{s}_{p,\infty}}+\|u^{\rm NS}_\tau\|_{B^{s}_{p,\infty}}\g)\dd\tau\\
\leq&~ Ct^2\delta^3.
\end{align*}
Then we complete the proof of Proposition \ref{pr3}.
\end{proof}
\subsection{Non-Convergence}\label{sec3.4}
In this subsection, we shall prove the non-convergence of the Navier-Stokes equations toward the Euler equations in the inviscid limit under the framework of endpoint Besov spaces $B^{s}_{p,\infty}$.

From \eqref{yh1} and \eqref{yh2}, it follows that
\bal\label{fj}
u^{\rm NS}_t-u^{\rm E}_t&=\left(e^{t\ep\Delta}u_0-u_0\right)+\mathcal{E},
\end{align}
where
$$\mathcal{E}:=-\int_0^te^{(t-\tau)\ep\Delta}\mathcal{P}\f(u^{\rm NS}_{\tau}\cdot \nabla u^{\rm NS}_{\tau}\g)\dd\tau+\int_0^t\mathcal{P}\f(u^{\rm E}_{\tau}\cdot \nabla u^{\rm E}_{\tau}\g)\dd\tau.$$
We decompose the error term $\mathcal{E}$ into four terms
\bbal
\mathcal{E}&=-\int_0^te^{(t-\tau)\ep\Delta}\mathcal{P}\left(e^{\tau\ep\Delta}u_0\cd\na e^{\tau\ep\Delta}u_0-u_0\cd\na u_0\right)\dd\tau\nonumber\\
&\quad-\int_0^t\left(e^{\tau\ep\Delta}-\mathrm{Id}\right)\mathcal{P}\left(u_0\cd\na u_0\right)\dd\tau+\int_0^t\mathcal{P}\f(u^{\rm E}_{\tau}\cdot \nabla u^{\rm E}_{\tau}-u_0\cd\na u_0\g)\dd\tau\nonumber\\
&\quad+\int_0^te^{(t-\tau)\ep\Delta}\mathcal{P}\left(e^{\tau\ep\Delta}u_0\cd\na e^{\tau\ep\Delta}u_0-u^{\rm NS}_{\tau}\cdot \nabla u^{\rm NS}_{\tau}\right)\dd\tau,
\end{align*}
or equivalently,
\bal\label{fj-1}
\mathcal{E}=\sum_{i=1}^4\mathbf{I}_i,
\end{align}
where
\bbal
&\mathbf{I}_1:=-\int_0^te^{(t-\tau)\ep\Delta}\mathcal{P}\left(e^{\tau\ep\Delta}u_0\cd\na e^{\tau\ep\Delta}u_0-u_0\cd\na u_0\right)\dd\tau,\\
&\mathbf{I}_2:=-\int_0^t\left(e^{\tau\ep\Delta}-\mathrm{Id}\right)\mathcal{P}\left(u_0\cd\na u_0\right)\dd\tau,\\
&\mathbf{I}_3:=-\left[u^{\rm E}_t-u_0+t\mathcal{P}\left(u_0\cd\na u_0\right)\right],\\
&\mathbf{I}_4:=u^{\rm NS}_t-e^{t\ep\Delta}u_0+\int_0^te^{(t-\tau)\ep\Delta}\mathcal{P}\left(e^{\tau\ep\Delta}u_0\cd\na e^{\tau\ep\Delta}u_0\right)\dd\tau.
\end{align*}
From \eqref{fj} and \eqref{fj-1}, using the triangle inequality and Proposition \ref{pr0}, we have
\bal\label{lf1}
\f\|u^{\rm NS}_t-u^{\rm E}_t\g\|_{{B}^{s}_{p,\infty}}
&\geq 2^{ns}\f\|{\De}_n\left(e^{t\ep\Delta}u_0-u_0\right)\g\|_{L^p}-2^{ns}\sum_{i=1}^4\f\|{\De}_n\mathbf{I}_i\g\|_{L^p}\nonumber\\
&\geq \delta2^{n}t(c-Ct2^n)-C2^{n}\f(\f\|\mathbf{I}_1\g\|_{\dot{B}^{s-1}_{p,\infty}}+\f\|\mathbf{I}_2\g\|_{\dot{B}^{s-1}_{p,\infty}}\g)\nonumber\\
&\quad-C2^{2n}\f(\f\|\mathbf{I}_3\g\|_{\dot{B}^{s-2}_{p,\infty}}+\f\|\mathbf{I}_4\g\|_{\dot{B}^{s-2}_{p,\infty}}\g).
\end{align}
Using Proposition \ref{pr1}, we have
\bbal
\f\|\mathbf{I}_1\g\|_{\dot{B}^{s-1}_{p,\infty}}+\f\|\mathbf{I}_2\g\|_{\dot{B}^{s-1}_{p,\infty}}\leq Ct\delta^2.
\end{align*}
Using Propositions \ref{pr2}-\ref{pr3}, we have
\bbal
\f\|\mathbf{I}_3\g\|_{\dot{B}^{s-2}_{p,\infty}}+\f\|\mathbf{I}_4\g\|_{\dot{B}^{s-2}_{p,\infty}}\leq Ct^2\delta^3.
\end{align*}
Inserting the above into \eqref{lf1}, and picking $t=\delta2^{-n}$ with the above $\delta$, we deduce that for large $n$
\bbal
\f\|u^{\rm NS}_t-u^{\rm E}_t\g\|_{{B}^{s}_{p,\infty}}&\geq \delta 2^nt\f(c-Ct2^n\g)-C\delta^2t2^n-C\delta^3t^22^{2n}\\
&\geq \delta^2 \f(c-C\delta\g)-C\delta^3-C\delta^5.
\end{align*}
Then, choosing $\delta$ sufficiently small, we deduce that
\bbal
\f\|u^{\rm NS}_{\ep}(t,u_0)-u^{\rm E}(t,u_0)\g\|_{B^{s}_{p,\infty}}&\geq c_0\delta^2=:\eta_0>0.
\end{align*}
Thus we complete the proof of \eqref{impor} in Theorem \ref{th2}. Notice that $t\to 0^+$ and $\ep\to0^+$ as $n\to\infty$, one has
$$
\liminf_{\ep\downarrow 0}\f\|u^{\rm NS}_{\ep}(t,u_0)-u^{\rm E}(t,u_0)\g\|_{B^{s}_{p,\infty}}\geq  \eta_0>0,
$$
which implies the non-convergence of Theorem \ref{th2}. In fact, if $\lim\limits_{\ep\downarrow 0}\f\|u^{\rm NS}_{\ep}(t,u_0)-u^{\rm E}(t,u_0)\g\|_{L_T^\infty B^{s}_{p,\infty}}=0$ were true for any small $T$, this will lead to a contradiction.
{\hfill $\square$}

\section*{Acknowledgments}
The authors would like to express their gratitude to the anonymous referees for valuable suggestions and comments which greatly improved the paper.

\section*{Declarations}

\noindent\textbf{Availability of data and materials}\\ No data was used for the research described in the article.
\vspace*{1em}

\noindent\textbf{Conflict of interest}\\
The authors declare that they have no conflict of interest.
\vspace*{1em}

\noindent\textbf{Funding}\\
J. Li is supported by the National Natural Science Foundation of China~(No.12161004),
Innovative High end Talent Project in Ganpo Talent Program (No.gpyc20240069), Training Program for Academic and Technical Leaders of Major Disciplines in Ganpo Juncai Support Program (No.20232BCJ23009).

\end{document}